\documentclass[11p,reqno]{amsart}
\usepackage{amsmath}
\usepackage{amsfonts}
\usepackage{amssymb}
\usepackage{graphicx}
\usepackage{color}
\newtheorem{theorem}{Theorem}[section]
\newtheorem*{theorem*}{Theorem}
\newtheorem{lemma}{Lemma}[section]
\newtheorem{corollary}[theorem]{Corollary}
\newtheorem{proposition}{Proposition}[section]

\newtheorem{problem}[theorem]{Problem}

\def\Ric{\text{Ric}}

\def\R{\Bbb R}

\def\Ric{\operatorname{Ric}}

\def \R {{\Bbb R}}

\def\Ric{{\operatorname{Ric}}}

\numberwithin{equation}{section}

\begin{document}
	\title[Ricci pinching]{Is a complete Riemannian manifold with positively pinched Ricci curvature compact}

\author{Lei Ni}

\address{Lei Ni. Zhejiang Normal University, Jinhua, Zhejiang, China, 321004}
\address{department of Mathematics, University of California, San Diego, La Jolla, CA 92093, USA}
\email{lni.math.ucsd@gmail.com}


\subjclass[2010]{}

\begin{abstract} A result of R. Hamilton asserts that any convex hypersurface in an Euclidian space with pinched second fundamental form must be compact. Partly inspired by this result,  twenty years ago, in \cite{Ancient}, Remark 3.1 on page 650,  the author formulated a problem asking if a complete Riemannian manifold with positively pinched Ricci curvature must be compact. There are several recent progresses, which are all rigidity results concerning the flat metric except the special case for the steady solitons.   In this note we provide a detailed alternate proof of Hamilton's result, in view of the recent proof via the mean curvature flow requiring additional assumptions and that the original argument by Hamilton does lack  of complete details. The proof uses a result of the author in 1998 concerning quasi-conformal maps.  The proof here allows a generalization as well. We dedicate this article to commemorate R. Hamilton, the creator of the Ricci flow, who also made fundamental contributions to many other geometric flows.
\end{abstract}

\maketitle

\section{Introduction}
In the study of the convergence and asymptotic shapes of the mean curvature flow, motivated by a simplified alternate proof of Huisken's convergence result,  Richard Hamilton \cite{Hamilton94} proved the following result.

\begin{theorem}[Hamilton, 94]\label{thm:ham} Let $M^n (n\ge 2)$ be a smooth strictly convex complete hypersurface
bounding a region in $\mathbb{R}^{n+1}$. Suppose that its second fundamental form
is $\epsilon$-pinched in the sense that
\begin{equation}\label{eq:pinch1} II_{ij} \ge  \epsilon H g_{ij}
\end{equation}
where $g_{ij}$ is the induced Riemannian metric of $M$, $II_{ij}$ the second fundamental form,
and its trace $H$ is the mean curvature, for some $\epsilon > 0$. Then $M^n$ is compact.
\end{theorem}

As explained in \cite{Hamilton94}, the result indeed simplifies the proof of the convergence to the round sphere, a result of Huisken. Another simplification of Huisken's result was via the scalar PDE of mean curvature flow in terms of  the support function and was obtained in \cite{Andrews} around the same time (the method also applies to a vast class of flows including the ones by $1/n$ root of the Gauss curvature and square root of the scalar curvature).
The recent detailed proof of the above result in \cite{BLL} does not provide such a simplification to the convergence since it uses  the study of the mean curvature flow (in addition it requires the assumption of the boundedness of the second fundamental form). We first provide a self-contained proof using a result of \cite{Mich}. This approach was known to the author before 2004, the writing of \cite{Ancient}. First we recall the following result of \cite{Mich}.

\begin{theorem}[Ni, 98]\label{thm:Ni98} Let $M^n $ be a complete Riemannian manifold satisfying that
\begin{equation}\label{eq:volume1}
\lim_{r\to \infty} \frac{V(o, r)}{r^n}=0
\end{equation}
where $V(o, r)$ denotes the volume of the ball of radius $r$ centered at $o\in M$. And let $N^n$ be a Riemannian manifold with $\lambda_1(N) > 0$, where $\lambda_1(N)$   is the lower bound of the
spectrum of the Laplacian–Beltrami operator. Then there is no quasi-conformal
diffeomorphism from $M$ into $N$.
\end{theorem}
Note here $N$ is not assumed to be complete. We say that the lower bound of the
spectrum of the Laplacian–Beltrami operator $\lambda_1(N)>0$  if for such $\lambda_1$
$$
  \int_N |\nabla \varphi|^2 \ge \lambda_1 \int_N \varphi^2, \forall \varphi \in C_0^\infty (N).
$$

The link between Theorem \ref{thm:ham} and Theorem \ref{thm:Ni98} is the Gauss map $\nu$ of $M$, which is a quasi-conformal diffeomorphism between $M^n$ and the round sphere $\mathbb{S}^n\subset \mathbb{R}^{n+1}$, under the assumption of (\ref{eq:pinch1}). The additional input is Hamilton's conformal deformation via an `entropy function'.

The first purpose of the paper is to provide a detailed alternate proof of Theorem \ref{thm:ham} via Theorem \ref{thm:Ni98}. Since the proof of Theorem \ref{thm:Ni98} was sketched in \cite{Mich}, we  include the details here as well as an extension. The original motivation of Theorem \ref{thm:Ni98} is the study of quasi-conformal harmonic diffeomorphisms of Li, Tam and Wang \cite{LTW}. The key component of Theorem \ref{thm:Ni98} is to drop the assumption on the harmonicity of the map.

Partially inspired by Hamilton's result  we proposed the following problem in \cite{Ancient},  Remark 3.1 of page 650.

\begin{problem}\label{prob:1} Let $M^n$ ($n\ge 3$) be a complete Riemannian manifold whose nonnegative Ricci curvature satisfies that
\begin{equation}\label{eq:ric-pinching}
\Ric_{ij}\ge \epsilon S g_{ij}
\end{equation}
with $\epsilon>0$ and $S$ being the scalar curvature. Is $M^n$ either Ricci flat or  compact? A weaker version is to prove that any complete Riemannian manifold with positive Ricci curvature and (\ref{eq:ric-pinching}) must be compact.
\end{problem}

Besides the above theorem of Hamilton, the other motivation of the  problem is the Bonnet-Myer's theorem on the closeness of a Riemannian manifold with its Ricci curvature being bounded from below by a positive constant.
After we finished the draft of \cite{Ancient} in 2003-2004, whose results were reported in the  December, 2003 ICCM at  Hong Kong,  B. Chow informed the author that for $n=3$ the problem was due to R. Hamilton. Given that the result implies  Hamilton's convergence result of the Ricci flow for three manifolds with positive Ricci curvature, an alternate approach (without using the Ricci flow) via the above problem in dimension three must have been what R. Hamilton had in mind. Since the Ricci flatness does not imply the flatness for manifolds with dimension greater than three, the problem above is of completely different nature for dimensions greater than three. Therefore we shall not put emphasis on the recent progresses of the three dimensional case \cite{DSS, LT2}. One may consult \cite{BMOP} for detailed account of various earlier development, particularly the relationship with the earlier works of \cite{Lott, Huisken}. Special cases for the steady and expanding solitons with nonnegative sectional curvature were obtained in \cite{Ancient}, Corollary 3.1. These were extended for corresponding K\"ahler-Ricci solitons in \cite{DZ} by Deng-Zhu removing the additional assumption of nonnegative sectional curvature.  The progresses after the appearance of \cite{Ancient} is discussed in a later section, where  we also recall that the answer to the problem is affirmative if $(M, g)$ is a gradient shrinking soliton, as a corollary of a main result in \cite{Ancient}. Gradient Ricci solitons provide  good testing cases for Problem \ref{prob:1}. We would also like to advocate it as one of problems in the geometry of nonnegative Ricci curvature, given the  tremendous advances on the structure of spaces of manifolds with lower Ricci curvature bound  in the last decades (with references too numerous to mention here) and  the result of \cite{Lock}.

\section{Preliminaries}
Let $M^n\subset \mathbb{R}^{n+1}$ be a smooth convex hypersurface. Assume that it is not compact.
By the Hadamard-Stoker theorem and its high dimensional generalization (cf. Theorem 6.5 of \cite{MR}, and page 17 of \cite{Bake}), one can express $M$ as the graph of a convex (smooth) function $f(x)$, defined over an open convex domain $\Omega \subset \mathbb{R}^n$. Namely $M=\{(x, f(x))\, |\, x\in \Omega\}$ with $f(x)$ being a convex function. Without the loss of generality, by picking a support hyperplane passing a boundary point,  we may assume that $0\in \Omega$,  $f\ge3$, $f(0)=3$ and $\nabla f(0)=0$. We also use $(x, z)$ to denote the coordinate function of $\mathbb{R}^{n+1}=\mathbb{R}^n \times \mathbb{R}$. The induced metric is expressed as $ds^2=\left(\delta_{ij}+\frac{\partial f}{\partial x^i}\frac{\partial f}{\partial x^j}\right) dx^i dx^j$.

In order to apply Theorem \ref{thm:Ni98} the following lemma from \cite{Hamilton94} is needed.

\begin{lemma}[Hamilton]\label{lem:21}
The conformally equivalent metric $\tilde{g}_{ij}=\frac{1}{f^2\log ^2f} g_{ij}$
is complete with finite volume. Here $g_{ij}$ is the induced complete metric on $M$ described as the above.
\end{lemma}
Note that the result holds for any unbounded strictly convex hypersurface. One may call the conformal change $\tilde{g}$ the `entropy'-conformal metric.

\begin{proof} Since the proof in \cite{Hamilton94} lacks of complete details and sometime it is hard to figure out what the argument really means we include a detailed proof here. In particular the argument here appeals to the  co-area formula and the Kubota’s inequality \cite{Kubota}  besides the ideas from \cite{Hamilton94}. First we pick a small $\eta_1>0$ such that
$
\overline{B(0, \eta_1)}\subset \Omega.
$
By the convexity of $f$ there exists a small $\eta>0$ such that outside $B(0, \eta_1)$, $ \frac{\partial f}{\partial x^i}\ge \eta>0$. In fact  for any unit vector $V$, we have that
\begin{equation}\label{eq:convex-est1}
\langle \nabla f, V\rangle \ge \eta, \, \mbox{ at }  x=V|x|\in (\overline{B(0, \eta_1)})^c.
\end{equation}
Indeed along the curve $\gamma(s)=Vs$ from $0$ to $x=Vs_0\in (\overline{B(0, \eta_1)})^c$ (with $s_0>\eta_1$)
\begin{eqnarray*}
\langle \nabla f, V\rangle &=& \frac{d}{ ds}f(\gamma(s))|_{s_0}\\
&=&\int_0^{s_0} \frac{d^2}{ ds^2}f(\gamma(s))\, ds\\
&\ge& \eta_0\eta_1 (=:\eta)
\end{eqnarray*}
if $\eta_0$ is the lower bound of the Hessian of $f$ over $\overline{B(0, \eta_1)}$.

The convex function $f$ is locally Lipschitz, hence over any closed ball $\overline{B}\subset \Omega$ the function $f$ is uniformly bounded. It is  easy to see from the completeness of $M$ that as $x\to \partial \Omega$, $f(x)\to \infty$. Indeed, the lower bound of the outer radial directional derivative outside $B(0, \eta_1)$ implies the result for $x_i\to  \partial \Omega$ with $|x_i|\to \infty$. For the case that $\{|x_i|\}$ is bounded, namely $x_i\to x_\infty \in \partial \Omega$, the completeness of $M$ implies the result. Now we show that $\tilde{g}$ is complete.\footnote{This is shown to us by Yongjia Zhang.}

For a sequence of point $x_i\to \partial \Omega$, let $\gamma$ be any curve joining $0$ to $x_i$.
Then
\begin{eqnarray*}
\tilde{L}(\gamma)&=& \int_0^1 \frac{\sqrt{ |\dot{\gamma}|^2+\langle \nabla f, \dot{\gamma}\rangle^2}}{f\log f}\, ds\\
&\ge&  \int_0^1 \frac{|\langle \nabla f, \dot{\gamma}\rangle|}{f\log f}\, ds\\
&\ge& \left|\int_0^1 \frac{\langle \nabla f, \dot{\gamma}\rangle}{f\log f}\, ds\right|\\
&\ge& \log\log f(x_i)-\log\log f(0)\to \infty.
\end{eqnarray*}

By the area formula (cf. Theorem 3.4 of \cite{EG}) given any $D\subset \Omega$, the area of the graph over $D$, namely $M\cap \{(x, f(x))\, |\, x\in D\}$, is given by the area formula
$$
V(G(D))=\int_D \sqrt{ 1+|\nabla f|^2}\, dx
$$
and with respect to the deformed metric $\tilde{g}$, the area is given by
$$
\tilde{V}(G(D))=\int_D \frac{\sqrt{ 1+|\nabla f|^2}}{(f\log f)^n}\, dx.
$$
  Since radial derivative of $f$ is positive and bounded away from $0$ except at $x=0$, $f(0)=3$, and $f$ can only increase by a bounded amount along radial curves within $B(0, \eta_1)$, by possibly shrinking $\eta_1$, we may choose a generic,  $\delta_1>0$, such that $\{f> 3+\delta_1\}$ has no intersection with $B(0, \eta_1)$. Moreover, we  choose $\delta_1$ relatively small so that $ \overline{\{f\le 3+\delta_1\}}$ is relatively compact in $\Omega$ . In fact since $\overline{B(0, \eta_1)}\subset \Omega,$ we first pick $\eta_2>\eta_1$ such that $\overline{B(0, \eta_2)}\subset \Omega$, and then  $\delta_1$ to satisfy
  $$\sup_{x\in \partial B(0, \eta_2)}f(x)=3+\delta_1.$$
  Such $\delta_1$ satisfies all the conditions specified above. Indeed we may set a $\eta_3>0$ with $\overline{B(0, \eta_3)} \subset \Omega$ first,  and then pick $\delta_2>0$ such that $ \overline{\{f\le 3+\delta_2\}}\subset B(0, \eta_3)$ by letting $3+\delta_2< \inf_{\partial B(0, \eta_3)}$, and then choose $\eta_1< \eta_2 <\eta_3$ and $\delta_1\in (0, \delta_2)$ as above.

  Write  $\tilde{V}(G(M))=\tilde{V}(G(\{f\le 3+\delta_1\}))+\tilde{V}(G(\{f> 3+\delta_1\})):=I+II$. By the choice of $\delta_1$, $|\nabla f|\ge \eta>0$ over $\{f> 2+\delta_1\}$. By the co-area formula \cite{EG} (note that $\nabla f\ne 0$ here) we write the second term above  as
$$
II=\int_{3+\delta_1}^\infty \int_{\{f=t\}} \frac{\sqrt{ 1+|\nabla f|^2}}{|\nabla f|}\cdot \frac{1}{(t\log t)^n}\, d\mathcal{H}^{n-1}\, dt.
$$
Hence, the lower estimate of $|\nabla f|$ yields
\begin{eqnarray*}
II&\le& \left(1+\frac{1}{\eta}\right)\int_{3+\delta_1}^\infty \int_{\{f=t\}} \frac{1}{(t\log t)^n}\, d\mathcal{H}^{n-1}\, dt\\
&=&\left(1+\frac{1}{\eta}\right)\int_{3+\delta_1}^\infty \frac{1}{(t\log t)^n} A(t)
\end{eqnarray*}
where $V(t)$ is the volume of the sublevel set $\Omega_t=\{ x\, |\, f(x)\le t\}$ and $A(t)$ is the area of the boundary $\partial \Omega_t$.
Note that for $t$ large the diameter of $\Omega_t=\{ x\, |\, f(x)\le t\}$ can be estimated by $Ct$ for some $C$ independent of $t$,   due to the estimate (\ref{eq:convex-est1}) outside a relatively compact small ball. Hence by the Kubota’s inequality \footnote{The Kubota's inequality asserts that for a convex body with diameter $D$ the area of its boundary has a sharp upper bound given by $\omega_{n-1}\left(\frac{D}{2}\right)^{n-1}$. Here $\omega_{n-1}$ is the area of $\mathbb{S}^{n-1}$. The result can be derived  from the Kubota's formula (cf. (31) of 19.4.2 of \cite{BZ}) and the isodiametric inequality, a Bieberbach's inequality (cf. Theorem 1 on page 69 of \cite{EG}).} (cf. \cite{Kubota} and page 409 of \cite{Sch})
\begin{equation}\label{eq:vol1}
A(t)\le C' t^{n-1}
\end{equation}
for some $C'$ which is independent of $t$. Therefore we have the estimate
\begin{eqnarray*}
II\le C(f, n, \eta) \int_{3+\delta_1}^\infty \frac{1}{t (\log t)^n}\, dt
< \infty.
\end{eqnarray*}
The estimate then implies the finiteness of the volume of the conformally entropy-deformed metric $\tilde{g}$.
\end{proof}

We also need another result on the Gauss map of a convex non-compact hypersurface in $\mathbb{R}^{n+1}$.

\begin{lemma}\label{lemma:wu}
Let $M^n$ be a complete noncompact orientable smooth hypersurface
in $\mathbb{R}^{n+1} (n \ge 2)$ with nonnegative but not identically zero sectional curvature.
Let $\nu: M \to \mathbb{S}^n$  be the spherical (Gauss) map. Then  after a congruence of $\mathbb{R}^{n+1}$, $\nu(M)$ lies in a closed hemisphere of $\mathbb{S}^n$.
\end{lemma}
\begin{proof} We only need the result for smooth strictly convex hypersurfaces, which is perhaps well-known. In any case  a proof can be found in \cite{Wu}, even for non-smooth convex hypersurfaces.
\end{proof}

\section{Proof of theorems}

The proof of Theorem \ref{thm:ham} is via contradiction argument.
First by the result of the last section we have the following corollary.

\begin{corollary}\label{coro: 1}
 Let $M^n$ be a smooth strictly convex complete hypersurface
bounding a region in $\mathbb{R}^{n+1}$. Suppose that its second fundamental form
is $\epsilon$-pinched in the sense of (\ref{eq:pinch1}).
If $M$ is not compact,  then $V(M, \tilde{g})<\infty$ and $\nu: (M, \tilde{g}) \to \overline{\mathbb{S}^n}_+$ is a quasi-conformal map.
\end{corollary}

Now Theorem \ref{thm:ham} follows from Theorem \ref{thm:Ni98} and the above corollary. Hence we only need to supply the proof of Theorem \ref{thm:Ni98} since it is clear that $\lambda_1\left(\mathcal{N}_\delta(\mathbb{S}^n_+)\right) > 0$, where $\mathcal{N}_\delta(\mathbb{S}^n_+)$ is a $\delta$-neighborhood of the closed upper semi-sphere.

To prove Theorem \ref{thm:Ni98} we also argue by contradiction. Let $\nu: (M, \tilde{g})\to N$ be such a quasi-conformal map. First  for a cut-off function $\phi$ on $M$ and let $\varphi=\phi(\nu^{-1})$ which is defined on $\nu(M)$ and has a compact support since $\nu$ is a diffeomorphism from $M$ to $\nu(M)$. By the assumption and Lemma 2.2 of \cite{Mich} we have that
\begin{equation}\label{eq:31}
A(n, \lambda_1) \int_{N} |\nabla^N \varphi|^n\, dv_N\ge  \int_N \varphi^n\, dv_N.
\end{equation}
 On the other hand, let $\kappa_i$ be the singular value of $d\nu$ (at any given point $x$) with $\kappa^2_1\ge \cdots\ge \kappa_n^2$ (with respect to $\tilde{g}$ and the metric on $N$). Then the pinching condition (\ref{eq:pinch1}) implies that $\frac{\kappa_i}{\kappa_j}\ge \epsilon$. Hence
 \begin{equation}\label{eq:help1}
 \int_{N} |\nabla^N \varphi|^n\, dv_N\le A''\int_M |\tilde{\nabla}^M \phi|^n\, d\tilde{v}_M
 \end{equation}
where $A'=A'(n, \epsilon)>0$. Indeed, choose two normal coordinates centered at $x$ and $\nu(x)$ respectively,  such that $d\nu(\frac{\partial}{\partial x^i})=\kappa_i \frac{\partial}{\partial y^i}$.  Point wisely, by abusing the notation involving the change of the volume form in terms of the change of the coordinates, we have at $x$ that
\begin{eqnarray*}
 |\tilde{\nabla}^M \phi|^n\, d\tilde{v}_M &=&\left(\sum \left(\frac{\partial \phi}{\partial x^i}\right)^2\right)^{n/2}d\tilde{v}_M\\
 &=& \left(\sum \left(\frac{\partial \varphi}{\partial y^i}\right)^2 \kappa_i^2\right)^{n/2} \frac{1}{\kappa_1\cdots\kappa_n} dv_N\\
 &\ge& \epsilon^n |\nabla^N \varphi|^n\, dv_N.
\end{eqnarray*}
Hence $A''$ can be chosen as $ \left(\frac{1}{\epsilon}\right)^n$.
If we choose $\phi$ to be a function  supported in $B(o, 2R)$ (with $o$ being a fixed point in $M$)  and being $1$ in $B(o, R)$ and $|\nabla^M \phi|\le {C}{R}$ for some absolute constant $C$, it then implies  following  the estimate,  for $R\ge 1$,
\begin{eqnarray*}
\int_{B(0, 1)}\left(\Pi_{i=1}^n \kappa_i\right) d\tilde{v}_M&\le& \int_M \phi^n \left(\Pi_{i=1}^n \kappa_i\right) d\tilde{v}_M\\
&=&\int_N \varphi^n\, dv_N \\
&\le& A(n, \lambda_1) \int_{N} |\nabla^N \varphi|^n\, dv_N\\
&\le& A(n, \lambda_1) A'' \int_M |\tilde{\nabla}^M \phi|^n\, d\tilde{v}_M \\
&\le& A' \frac{\tilde{V}(o,2R)}{R^n}\\
&\to& 0, \mbox{ as } R\to \infty.
\end{eqnarray*}
Here $A'=A(n, \lambda_1) A''$. This is a contradiction since the left hand side above is positive.

\section{A generalization of Theorem \ref{thm:Ni98} and testing cases for the Problem \ref{prob:1}}

For Bonnet-Myer's theorem, Calabi proved a generalization. For $(M^n, g)$ a complete Riemannian manifold, if we denote the lower bound of the Ricci curvature on $\partial B(o, r)$ by $k(r)>0$, the following result was proved in \cite{Calabi}

\begin{theorem}\label{thm:Calabi} [Calabi] Assume that
$$
\limsup_{a\to \infty} \left(\int_0^a \sqrt{k(s)}\, ds -\frac{1}{2}\log a\right)=\infty.
$$
Then $M$ must be compact.
\end{theorem}

Motivated by this and the theorem we just proved, for a diffeomorphism $\nu:M \to N$, we introduce
$$
\epsilon(r)=\inf_{x\in \partial B(o, r)} \{ \min_{i, j\in \{1, \cdots, n\}}  \frac{\kappa_i(x)}{\kappa_j(x)}\}.
$$
Here $\kappa_1\ge \cdots\ge \kappa_n$ are the singular values of $d\nu$.

The proof of the last section can be adapted to prove the following result.

\begin{proposition} \label{prop:31} If $(M^n, \tilde{g})$ is a complete Riemannian manifold with finite volume. Assume that $(N, g^N)$ has a positive lower bound on the spectrum the Laplacian–Beltrami operator, or more generally satisfies one of the following two Sobolev type inequalities
\begin{eqnarray}\label{eq:Sob1}
\left(\int_N |\varphi|^q \, dv_N\right)^{p/q} &\le& C(p, q) \int_N |\nabla \varphi|^p\, dv_N, \mbox{ with } 0<p< n, 0<q< \frac{np}{n-p},\\
 \mbox { or }\left(\int_N |\varphi|^q \, dv_N\right)^{n/q} &\le& C(n, q) \int_N |\nabla \varphi|^n\, dv_N, \mbox{ with } q>0\label{eq:Sob2}
\end{eqnarray}
for $\varphi \in C_0^\infty(N)$.
 Then there is no diffeomorphism $\nu: M\to N$ satisfying that
$$
\lim_{R\to \infty} \epsilon(R) R=\infty.
$$
\end{proposition}
\begin{proof} Under the Sobolev inequality (\ref{eq:Sob1}) we write $\psi=|\varphi|^\gamma$, with $\gamma=\frac{pn}{pn-q(n-p)}$. It is easy to check that $\gamma q =(\gamma-1)\frac{pn}{n-p}$. Now apply (\ref{eq:Sob1}) to $\psi$ and the H\"older's inequality we then have that
\begin{eqnarray*}
\left(\int_N |\varphi|^{\gamma q}\right)^{p/q}&\le& C\int_N |\varphi|^{(\gamma-1)p}|\nabla \varphi|^p\\
&\le& C' \left(\int_N |\nabla \varphi|^n\right)^{\frac{p}{n}}\left(\int_N |\varphi|^{(\gamma-1)\frac{pn}{n-p}}\right)^{\frac{n-p}{n}}.
\end{eqnarray*}
Hence we have the following Sobolev type inequality
\begin{equation}\label{eq:Sob-n}
\left(\int_N |\varphi|^{\gamma q}\, dv_N\right)^{\frac{n}{q}-\frac{n-p}{p}}\le C''(p, q, n)\int_N |\nabla \varphi|^n\, dv_N
\end{equation}
for any $\varphi\in C_0^\infty(N)$.

Now we apply the argument of last section to $\phi\in C_0^\infty(M)$, which equals to  $1$ in $B(o, R)$, and has a compact support in $B(o, 2R)$ with $\nabla^M \phi|\le \frac{C}{R}$ in $B(o, 2R)\setminus B(o, R)$, and $\nabla \phi=0$ in $B(o, R)$. Let $\varphi(y)=\phi(\nu^{-1}(y))$ which is a function with a compact support within $\nu(B(o, 2R))$.   Now (\ref{eq:help1}) can be improved to the following estimate:
\begin{equation}\label{eq:help2}
\int_{N} |\nabla^N \varphi|^n\, dv_N\le \frac{1}{\inf_{r\in [R, 2R]} \epsilon^n(r)}\int_M |\tilde{\nabla}^M \phi|^n\, d\tilde{v}_M.
\end{equation}
Combining this with (\ref{eq:Sob-n}) we have for $R\ge 1$
 \begin{eqnarray*}
\left(\int_{B(0, 1)}\left(\Pi_{i=1}^n \kappa_i\right) d\tilde{v}_M\right)^{\frac{n}{q}-\frac{n-p}{p}}&\le&\left( \int_M \phi^{\gamma q} \left(\Pi_{i=1}^n \kappa_i\right) d\tilde{v}_M\right)^{\frac{n}{q}-\frac{n-p}{p}}\\
&=&\left(\int_N \varphi^{\gamma q}\, dv_N \right)^{\frac{n}{q}-\frac{n-p}{p}}\\
&\le& C''  \int_{N} |\nabla^N \varphi|^n\, dv_N\\
&\le& C'' \frac{1}{\inf_{r\in [R, 2R]} \epsilon^n(r)} \int_M |\tilde{\nabla}^M \phi|^n\, d\tilde{v}_M \\
&\le& C''' \frac{\tilde{V}(o,2R)}{([\inf_{r\in [R, 2R]} \epsilon(r)]\cdot R)^n}\\
&\to& 0, \mbox{ as } R\to \infty.
\end{eqnarray*}
This yields a contradiction, which proves the proposition for the case that (\ref{eq:Sob1}) is assumed on $N$. The proofs for the other two cases are similar.
\end{proof}

The special case of Problem \ref{prob:1} for three dimensional manifolds, namely the Hamilton's problem/conjecture,  has an affirmative answer through the important work of \cite{DSS, LT2}. One may consult \cite{BMOP} for detailed account of various earlier developments particularly the relations with the earlier works of \cite{Lott, Huisken}. However, the solution uses the study of Ricci flow extensively. Given the simplicity of the proof presented above to Theorem \ref{thm:ham}, the solution via \cite{DSS,LT2} does {\it NOT} seem to be the solution that R. Hamilton envisioned. More likely he would have wished that the result providing a simplified proof to his original convergence theorem of 1982, hence preferably has  a proof independent of the Ricci flow?

One should note that the pinching condition (\ref{eq:pinch1}) is preserved by the mean curvature flow. And similarly, it can be shown that the  condition (\ref{eq:ric-pinching})  is preserved by the Ricci flow for $n=3$, however is a condition most likely not preserved in dimensions greater than three. Hence it would be highly interesting if the Ricci flow method can be applied to study Problem \ref{prob:1}.

For the high dimensional case, the first progress was made by B. Wu with the author  in \cite{Ni-Wu},  however under a stronger condition  that the curvature operator is pinched and the curvature is bounded. The proof is to study the singularity analysis of Ricci flow. The stronger assumption of the curvature operator pinching is necessary here since one needs to fit the curvature operator pinching assumption into one of the Ricci flow invariant cone conditions of B\"ohm-Wilking \cite{BW}.  The proof also needs   Corollary 3.1 of \cite{Ancient}. We note that the pinching condition (\ref{eq:ric-pinching}) is one of the two conditions in the definition of an invariant (with respect to the Ricci flow)  cone condition of \cite{BW} B\"ohm-Wilking (cf. Lemma 3.2 of \cite{Ni-Wu}). In \cite{LT1} Lee-Topping made nice further progresses and obtained results under the complex sectional curvature pinching, as well as the so-called PIC1 pinching, with geometric additional assumptions. The PIC1 is a high dimensional generalization of the Ricci curvature, which is preserved under the Ricci flow. Note that the positivity of PIC1 and complex sectional curvature are  stronger than (or at the least different from) the positivity of the Ricci curvature in dimensions greater than three.  Moreover the results obtained so far are  all of  rigidity type results concerning the flat metric, while Problem \ref{prob:1} concerns the  flatness of the Ricci curvature  instead. The main result of \cite{LT1} as an excellent progress in terms of extending \cite{Ni-Wu}, particularly  removes the curvature bound condition,  however avoids Problem \ref{prob:1}.

Here we point out that if $M$ is a gradient Ricci shrinking soliton, then the answer to Problem \ref{prob:1} is positive. Namely

\begin{proposition}\label{prop:32} Let $(M, g, f)$ be a gradient Ricci shrinking soliton in the sense that
\begin{equation}\label{eq:shrinker}
\nabla^2_{ij} f+ \Ric_{ij} -\frac{1}{2}g_{ij} = 0.
\end{equation}
Here $f$ is the potential function. Assume further that it satisfies that (\ref{eq:ric-pinching}). Then either $(M, g)$ is Ricci flat (in fact isometric to $\R^n$) or $(M, g)$ is compact.
\end{proposition}

The key is the Proposition 1.1 of \cite{Ancient}, which provides a uniform lower bound of the scalar curvature for gradient shrinking solitons with nonnegative (nonzero) Ricci curvature.

\begin{proposition}\label{prob:ancient1-1}Let $(M,g, f)$ be a Ricci non-flat gradient shrinking soliton.
Assume that the Ricci curvature of M is nonnegative. Then there exists a
$\delta = \delta(M)$ with $1 \ge \delta > 0$,  such that
\begin{equation}\label{eq:scalar-lower}
 S(x)\ge \delta.
\end{equation}
\end{proposition}

Now Proposition \ref{prop:32} is a simple consequence of Proposition \ref{prob:ancient1-1} and the Bonnet-Myers' theorem. The flatness is elementary.  By the Ricci flatness and soliton equation we have  $g_{ij}=\nabla^2_{ij} f$.  The rest argument then is detailed in \cite{DZ}. One may also use the argument on page 94 of \cite{Ni-entropy}.
Letting the point $o$ be the point $\nabla f=0$, without the loss of the generality we may assume that $ f(o)=0$. Then integrating along the minimizing geodesics from $o$ to any point $x\in M$  as in Section 2 one can check that $f(x)=\frac{1}{2} r^2(x)$, where $r(x)$ denotes the distance function to $o$. While tracing the soliton equation we have $\Delta f=n$, which, after applying the Green's formula, implies that for $R>0$
$$
\frac{RA(o, R)}{V(o, R)}=n
$$
where $A(o, R)$ denotes the area of $\partial B(o, R)$. The equality case implies the flatness of $B(o, R)$ for all $R$. (The strict convexity of $f$ implies that $M$ is diffeomorphic to $\mathbb{R}^n$.)

For noncompact gradient steady and expanding solitons, Problem \ref{prob:1} is to show the Ricci flatness under the assumption (\ref{eq:ric-pinching}). This was verified in \cite{Ancient} under the additional assumption that the manifold has nonnegative sectional curvature. The details of the proofs were  in \cite{Ni-ancient-ALM} for the reason that the journal where \cite{Ancient} published  put a page limit for \cite{Ancient}. A natural question, as  testing cases of Problem \ref{prob:1}, is whether or not the result can be proved without the additional assumption that the manifold has nonnegative sectional curvature. In \cite{DZ}, Deng-Zhu proved  the Ricci flatness for the noncompact expanding gradient solitons for the K\"ahler case under the assumption (\ref{eq:ric-pinching}). They also proved the same result for the steady gradient solitons assuming that they are K\"ahler and the potential function $f$ has a critical point somewhere. The proof is based on the decay estimate established in \cite{Ni-ancient-ALM} and an estimate of the quotient of volume elements for the K\"ahler-Ricci flow (Theorem 2.1 of \cite{NT}) via the Green's function estimate of Li-Yau. In \cite{DZ}, the authors proved a rigidity result for the K\"ahler gradient expanding soliton, namely the soliton is Gaussian,  under the assumption
$$
k(o, r):=\frac{1}{V(o, r)}\int_{B(o, r)} S(y)\, dv(y)\le \frac{\epsilon(r)}{1+r^2}
$$
with $\lim_{r\to \infty}\epsilon(r)\to0$. In fact, Theorem 2.1 of \cite{NT} implies the following result of the Ricci flatness without the pinching condition, but with some others.

\begin{theorem}[Ni-Tam]\label{thm:ricci-flatness}
 If $(M, g_0)$ is complete K\"ahler manifold with nonnegative Ricci curvature. Assume that there exists a K\"ahler-Ricci flow $g(t)$ with $g(0)=g_0$,  for all time $t\ge 0$, and $g(t)$ has nonnegative Ricci curvature. Assume further that the scalar curvature $S_0(y)$ of $g_0$ satisfies that
\begin{equation}\label{eq:thm41}
 \frac{1}{V_0(x, r)}\int_{B_0(x, r)} S_0(y) dv_0(y)\le k(r), \mbox{ and } \lim_{R\to \infty} \frac{\int_0^R s k(s)\, ds}{\log R}=0,
\end{equation}
where $B_0(x, R)$ is the ball of radius $R$ centered at $x$, $V_0(x, R)$ its volume and $S_0(x)$ the scalar curvature with respect to $g_0$. In addition also assume that $S(x, t)\le Ct^{\alpha}, \forall t>>1$, for some $\alpha>0$ and the Harnack estimate $\frac{\partial}{\partial t} \left(tS(x, t)\right)\ge 0$ holds.
 Then $(M, g_0)$ is Ricci flat.

The same result holds if $S(x, t)\le C$, $\frac{\partial}{\partial t} S(x, t)\ge 0$ and the assumption (\ref{eq:thm41}) is replaced with
\begin{equation}\label{eq:thm42}
\lim_{R\to \infty}  R k(R)=0.
\end{equation}
\end{theorem}
\begin{proof} With the notation of \cite{NT}, it suffices to show that $-m(t)=o(\log(t))$. Then the assumed Harnack estimate finishes the argument.

The assumption of $S(x, t)$ implies a rough estimate that $-m(t)\le C t^{\alpha+1}, \forall t >>1$. The assumption of $k(s)$ and Theorem 2.1 of \cite{NT} implies that there is a constant $C_1=C_1(m)$,
$$
-m(t)\le C_1\left[\left( 1+\frac{t(1-m(t))}{R^2}\right)\int_0^R sk(s)\, ds-\frac{tm(t)(1-m(t))}{R^2}\right]
$$
for any $R>0$. Picking $R= t^{\alpha+3}$, with $t>>1$, applying the rough estimate $-m(t)\le C t^{\alpha+1}$ to  the right hand side above,  we have an improved estimate
$$
-m(t)\le C_1 \int_0^R sk(s)\, ds +o(1).
$$
The needed estimate $-m(t)=o(\log t)$ follows by noticing  $R= t^{\alpha+3}$ and the assumption that $\int_0^R s k(s)\, ds$ is of $o(\log(R))$.

For the second case, the proof is basically the same as that of \cite{DZ}. The stronger Harnack estimate $\frac{\partial}{\partial t} S(x, t)\ge 0$ implies that it is sufficient to show $-F(x, t)=o(t)$.

Let $R=\frac{t}{\epsilon(t)}$ for some $\epsilon(t)$ with $\lim_{t\to \infty}\epsilon(t)= 0$. Namely $R$ is slightly faster than $t$ as $t\to \infty$. We may assume that $\epsilon(t)$ is monotone decreasing function. Its precise form can be decided by the another function $\epsilon'(R)$ which is chosen so that
$$
\int_0^R s k(s)\, ds \le \epsilon'(R) R
$$
with $\lim_{R\to \infty}\epsilon'(R)\to 0$ and $\epsilon'(R)$ monotone non-increasing. Now we may simply let $\epsilon(t)=\sqrt{\epsilon'(t)}$.

Now we apply Theorem 2.1 of \cite{NT} and the rough estimate $-m(t)\le C t$ (which follows from $S(x, t)\le C$) to get an improved (or enforced) estimate
$$
-m(t)=o(t).
$$
This implies the claimed Ricci flatness via $S(x, 0)=0$.
\end{proof}

Variations of the above result include Corollary 2.1-2.5 of \cite{NT}.
This extends the result of \cite{DZ}  a little. All assumptions of the theorem, except the one for the average of the scalar curvature of the initial metric  hold for the case of K\"ahler-Ricci flow from a metric with nonnegative bisectional curvature and maximum volume growth, as well as  the expanding and steady gradient K\"ahler solitons with nonnegative Ricci curvature. Hence it implies a Ricci flatness result for them.  A result similar to Theorem \ref{thm:ricci-flatness} can be formulated for the intermediate cases with $\frac{\partial}{\partial t}(t^\theta S(x, t))\ge 0$ for some $\theta\in [0, 1]$.

Another testing case for Problem \ref{prob:1} is when there is a solution to the Poincar\'e-Lelong equation, namely a smooth function $u$ such that $\sqrt{-1}\partial \overline{\partial} u=\Ric$.

Related to the Ricci nonnegative K\"ahler manifold, a favorite problem of the author is to show the Liouville type theorem for plurisubharmonic functions on a complete K\"ahler manifold $M^m$ with nonnegative Ricci curvature, {\it namely any sub-$log(d(o, x))$ growth plurisubharmonic function $u$ must be constant.} In author's Ph. D. thesis of 1998, it was proved that $(\sqrt{-1}\partial \bar{\partial} u)^m=0$. This is also related to understanding the structure of manifolds whose Ricci curvature degenerates, but with some high ranks.

In summary,  the answer to Problem \ref{prob:1} is affirmative for all shrinking gradient solitons, and essentially all gradient K\"ahler steady and expanding solitons. Moreover for the case of K\"ahler-Ricci expanding solitons, Deng-Zhu \cite{DZ} proved that the manifold is isometric to the complex Euclidean space. Hence even though the answers to Problem \ref{prob:1} are affirmative for gradient Ricci (K\"ahler-Ricci) soltions (for steady gradient solitons one needs to assume additionally that there is a point where $\nabla f=0$), it is a rigidity result of the flat metric (the Gaussian soliton) for the case of expanding gradient solitons. The Ricci flatness can be obtained via the K\"ahler-Ricci flow, however assuming {\it the long time existence and a differential Harnack estimate} and others.

 For the solitons which are not K\"ahler, Theorem 20.2 of \cite{Ham-Sing} reduces the problem to the existence of a Ricci flat metric on $\mathbb{S}^{n-1}$. Namely {\it if there is no Ricci flat metric on $\mathbb{S}^{n-1}$, then Problem \ref{prob:1} has a positive answer for the gradient steady soliton assuming additionally the potential function $f$ with $\nabla f(o)=0$ at some point $o$}.  Compared with the proof in Sections 2 and 3  of Theorem \ref{thm:ham} the main difficulty for Problem \ref{prob:1} is lack of some  `convexity' associated with the nonnegativity of the Ricci curvature.

\bigskip

\noindent\textbf{Acknowledgments.} {The author thanks J. Wang, G. Liu for their interest in the Problem \ref{prob:1},  Y. Zhang for discussions and reading the first draft. He also thanks  L. Russell and J. Dunworth for   the Alpenglow of Tahquitz Peak at their beautiful retreat home, which inspired some of the arguments, and G. Zhang for the reference on the Kubota’s inequality.}

\end{document}